\documentclass[12pt,a4paper,reqno,parskip=full]{amsart}
\usepackage{amsmath}
\usepackage{amsfonts}
\usepackage{amssymb}
\usepackage{colonequals}
\usepackage{parskip}
\usepackage{tikz}
\usepackage{tikz-cd}
\usetikzlibrary{matrix,arrows,decorations.pathmorphing}

\begingroup
\makeatletter
\@for\theoremstyle:=definition,remark,plain\do{
\expandafter\g@addto@macro\csname th@\theoremstyle\endcsname{
\addtolength\thm@preskip\parskip}}
\endgroup

\numberwithin{equation}{section}
\theoremstyle{plain}

\newtheorem{theorem}[subsection]{Theorem}
\newtheorem{proposition}[subsection]{Proposition}
\newtheorem{lemma}[subsection]{Lemma}
\newtheorem{corollary}[subsection]{Corollary}

\newtheorem{question}[subsection]{Question}
\newtheorem{remark}[subsection]{Remark}

\theoremstyle{definition}

\newtheorem{definition}[subsection]{Definition}
\newtheorem{example}[subsection]{Example}

\renewcommand{\leq}{\leqslant}
\renewcommand{\geq}{\geqslant}
\newcommand{\eps}{\varepsilon}

\DeclareMathOperator{\modulo}{mod}
\DeclareMathOperator{\End}{End}
\DeclareMathOperator{\Id}{Id}

\DeclareMathOperator{\Ham}{Ham}
\DeclareMathOperator{\Sch}{Sch}

\def\AA{{\mathcal A}}
\def\AAA{{\mathfrak A}}
\def\BB{{\mathcal B}}
\def\BBB{{\mathfrak B}}

\def\N{{\mathbb N}}

\def\R{{\mathbb R}}
\def\S{{\mathbb S}}

\def\Z{{\mathbb Z}}

\begin{document}

\title{Soficity for monoids, semigroups, and general dynamical systems}

\author{Jan Cannizzo}

\address{Department of Mathematical Sciences\newline
\indent Stevens Institute of Technology\newline
\indent 1 Castle Point on Hudson\newline
\indent Hoboken, NJ 07030\newline}

\email{jan.c.cannizzo@gmail.com}

%\thanks{}

%\subjclass{}

\maketitle

\begin{abstract}
We examine several definitions of soficity for monoids obtained by generalizing various definitions of sofic groups. They are not all equivalent and include the definition recently introduced by Ceccherini-Silberstein and Coornaert. One of these definitions readily generalizes to semigroups (albeit in an arguably unsatisfying way), thus addressing a question asked by Kambites. We conclude by proposing a definition of soficity for a general dynamical system consisting of a semigroup acting by measure-preserving transformations on a probability space.
\end{abstract}

\section{Introduction}

Since its introduction, the notion of soficity has proved to be of fundamental interest in several branches of mathematics. \emph{Sofic groups} were first defined by Gromov~\cite{Gro}, who showed that they satisfy Gottschalk's surjunctivity conjecture in topological dynamics, and they were given their name by Weiss~\cite{Wei} shortly thereafter (the word ``sofic'' is derived from the Hebrew word for ``finite''). Sofic groups are closely linked to a number of other conjectures, notably Connes' embedding conjecture and the determinant conjecture in the theory of von Neumann algebras (see \cite{ES} and the survey of Pestov \cite{Pes} for more), and it was a breakthrough when Bowen~\cite{Bow} defined a notion of \emph{sofic entropy} for measure-preserving actions of sofic groups, thereby extending the classical entropy theory developed for $\Z$-actions by Kolmogorov and Sinai.

It was soon realized that the notion of soficity can be extended to objects other than groups. By formulating soficity in terms of the weak convergence of measures (so-called \emph{Benjamini-Schramm convergence}, introduced in \cite{BS}), it became possible to speak of \emph{sofic random graphs}---see \cite{AL}---and, mutatis mutandis, \emph{sofic random Schreier graphs} (see, for instance, \cite{Can} for a definition). Using this idea, Elek and Lippner~\cite{EL} went on to define the more general notion of a \emph{sofic equivalence relation}, and working in greater generality still, Dykema, Kerr, and Pichot~\cite{DKP} have recently defined \emph{sofic groupoids} (see also \cite{Bow2}).

In each of the aforementioned contexts, the following question---arguably the most na\"{i}ve question it is possible to ask---has resisted proof and remains open as of this writing.
\begin{question}\label{q1}
Does there exist a nonsofic object?
\end{question}
Although the general opinion seems to be that nonsofic objects should exist, finding an example of such an object has proven to be difficult. It was therefore remarkable when Ceccherini-Silberstein and Coornaert~\cite{CSC} recently defined \emph{sofic monoids} and went on to answer Question~\ref{q1} by showing that there does exist a nonsofic monoid. In fact their example is not a particularly exotic object: it is the well-known \emph{bicyclic monoid}, namely the monoid with presentation
\[
B=\langle a,b\mid ab=e\rangle,
\]
where $e\in B$ denotes the identity. Further work on sofic monoids was carried out by Kambites~\cite{Kam} who, building on \cite{CSC}, exhibited a large class of sofic monoids.

Yet Ceccherini-Silberstein and Coornaert's definition of a sofic monoid, although certainly natural from an algebraic point of view, can be said to mark a departure from other definitions of soficity. Soficity as it is defined for groups, random graphs, equivalence relations, and groupoids is very much connected with ergodic theory---in particular, the presence of an invariant measure---whereas this connection appears to be absent from Ceccherini-Silberstein and Coornaert's theory. In fact it is automatic in the theories of other sofic objects that the existence of an underlying invariant measure is a prerequisite for defining soficity in the first place; hence, it is natural to ask whether it is possible to develop a theory of sofic monoids that takes invariance into account.

As noted by Kambites, soficity as defined in \cite{CSC} ``imposes no restriction whatsoever on the internal complexity of the monoid outside the group of units'' (\cite{Kam}, p.~12). Indeed, Proposition~4.7 of \cite{CSC} implies that, though the bicyclic monoid $B$ is nonsofic, the monoid $\hat{B}\colonequals B\cup\{e'\}$ obtained by adjoining a new element $e'$ to $B$ and defining $e'$ to be the identity is sofic---a striking fact, given that the internal structure of $\hat{B}$ is identical to that of $B$ outside of the point $e'$. Kambites concludes by remarking that
\begin{quote}
If seeking applications in semigroup theory more widely, one is drawn to ask if there is an alternative, probably stronger, definition of a sofic monoid which also generalises sofic groups but exerts more control on the internal structure of the rest of the monoid. A natural test of whether a definition is satisfactory in this respect would be whether the resulting class is closed under the taking of monoid subsemigroups.... Such a definition, if found, is also likely to extend naturally to semigroups without an identity element (\cite{Kam}, p.~12).
\end{quote}
One of the aims of this paper is to explore what such a definition might look like, beginning with the observation that a sofic group (just like an \emph{amenable group}) may naturally be defined in terms of \emph{approximately invariant measures}. We do, in fact, arrive at an alternative definition of a sofic monoid which readily generalizes to semigroups, but a semigroup is sofic in this new sense if and only if it embeds into a sofic group---arguably a heavy-handed condition. On the other hand, the definition is natural in that it provides a clear connection with ergodic theory, and although many semigroups (including the bicyclic monoid) remain ``nonsofic'' according to this definition, they are nonsofic for obvious reasons---roughly speaking, because they are not \emph{invariant structures}. This same phenomenon appears in the theories of other sofic objects as well: Given a probability measure on the space of Schreier graphs of a finitely generated group which is not conjugation-invariant, the corresponding random Schreier graph might be said to be ``nonsofic,'' but this is rather uninteresting, as it is a consequence of the definition of a sofic random Schreier graph that its law is conjugation-invariant. The principle here is that one must first restrict one's attention to invariant structures, and only then ask whether they are sofic.

To conclude, we offer another way to relate semigroups and soficity, namely by passing from semigroups to their actions on spaces equipped with an invariant measure. We thus formulate a notion of a general \emph{sofic dynamical system}. This generalizes the notion of a sofic equivalence relation introduced by Elek and Lippner and may be an interesting direction for future research.

This paper is organized as follows: In Section~2, we review background material and single out four equivalent definitions of sofic groups, each of which generalizes in a natural way to monoids. In Section~3, we show how these various generalizitions---one of which is the definition of Ceccherini-Silberstein and Coornaert---relate to one another and, in particular, are not all equivalent. Finally, in Section~4, we introduce the notion of a general sofic dynamical system, provide a basic example of such a system, and pose several questions related to our definition.

\section{Sofic groups and sofic monoids}

Sofic groups are, speaking very roughly, groups which admit finite approximations. They are a common generalization of amenable and residually finite groups (although there exist sofic groups which are neither amenable nor residually finite) and, as we will see, can be defined in a number of equivalent ways. We first establish some preliminaries.

If $S$ is a semigroup and $\AA$ is a generating set for $S$, then the \emph{(left) Cayley graph} $\Gamma=\Gamma(S,\AA)$ of $S$ constructed with respect to $\AA$ is the graph whose vertex set is identified with $S$ and whose edge set consists of all triples $(s,a,as)$, where $a\in\AA$ is a generator and $s\in S$ is an arbitrary element. The edge $(s,a,as)$ is understood to be directed from $s$ to $as$, which are its initial and terminal points, respectively, and labeled with $a$. The distance $d_\AA(s,t)$ between two vertices $s,t\in\Gamma$ is defined to be the length of a shortest sequence of edges, or path,
\[
(x_0,a_0,y_0),\ldots,(x_n,a_n,y_n)
\]
such that $x_0=s$, $y_n=t$, and $y_i=x_{i+1}$ for all $0\leq i\leq n-1$, provided that such a path exists. We set $d_\AA(s,s)=0$ for all $s\in\Gamma$. Note that $d_\AA$ is in general not a metric, since in a general semigroup, $d_\AA(s,t)$ need not be defined (in fact it is possible for $\Gamma$ to be disconnected, e.g.\ if $S$ is a free semigroup of rank $n>1$).

Given a vertex $s\in\Gamma$, we define the ball of radius $r\in\N$ centered at $s$, denoted $B_r(\Gamma,s)$, or just $B_r(s)$ for short, to be the subgraph of $\Gamma$ induced by the vertices at distance less than or equal to $r$ from $s$. If $S$ is a group and $\AA$ is symmetric, so that $a\in\AA$ implies $a^{-1}\in\AA$, then $B_r(s)$ coincides with the ball of radius $r$ taken with respect to the usual graph metric on $\Gamma$, but as before, note that this fails to be true for a general semigroup.

If $S$ has an identity $e$, then it is natural to regard $\Gamma(S,\AA)$ as a \emph{rooted graph} $(\Gamma,e)$ by distinguishing the vertex $e$. Given two rooted graphs $(\Gamma,x)$ and $(\Delta,y)$ whose edges are directed and labeled with elements of a set $\AA$ (we do not assume here that $\Gamma$ and $\Delta$ are Cayley graphs), we define a graph homomorphism $\phi:\Gamma\to\Delta$ to be a graph homomorphism in the usual sense which also satisfies $\phi(x)=y$ and maps one edge to another in such a way that respects the direction and labeling. A graph homomorphism is a graph isomorphism if it is invertible.

If $X$ is a set, we denote by $\End(X)$ the semigroup of all endomorphisms (that is, self-maps) $f:X\to X$, where the binary operation is composition of functions. If $X$ is finite, recall that the \emph{(normalized) Hamming metric} on $\End(X)$ is defined to be
\[
d_{\Ham}(f,g)\colonequals\frac{1}{|X|}|\{x\in X\mid f(x)\neq g(x)\}|.
\]
Given a semigroup $S$, we now define a \emph{$(K,\eps)$-action} of $S$ on a finite set as follows (see also \cite{ES} and \cite{Pes}).
\begin{definition}\label{def1}
Let $S$ be a semigroup, $\eps>0$ a real number, and $K\subseteq S$ a finite subset of $S$. A \emph{$(K,\eps)$-action} of $S$ on a finite set $X$ is a function $\psi:K\to\End(X)$ such that the following hold:
\begin{itemize}
\item[i.] If $s,t,st\in K$, then $d_{\Ham}\left(\psi(st),\psi(s)\circ\psi(t)\right)\leq\eps$.

\item[ii.] If $S$ has an identity $e$ and $e\in K$, then $d_{\Ham}\left(\psi(e),\Id_X\right)\leq\eps$, where $\Id_X$ is the identity map on $X$.

\item[iii.] For all distinct $s,t\in K$, one has $d_{\Ham}\left(\psi(s),\psi(t)\right)\geq1-\eps$.
\end{itemize}
\end{definition}

It can be useful to think of a $(K,\eps)$-action as a finite model of a (free) semigroup action in which a certain amount of error is allowed. We are now in a position to define sofic groups and will in fact present four equivalent definitions.

\begin{definition}\label{soficgroup} The following are equivalent definitions of a sofic group.
\begin{enumerate}

\item A group $G$ with finite generating set $\AA$ is sofic if for any $r\geq0$ and any $\eps>0$, there exists a finite $\AA$-labeled graph $\Gamma$ such that
\[
\frac{|\{x\in\Gamma\mid B_r(x)\cong G_{r,\AA}\}|}{|\Gamma|}\geq1-\eps,
\]
where $B_r(x)$ is the ball of radius $r$ centered at $x\in\Gamma$ and $G_{r,\AA}$ is the ball of radius $r$ centered at the identity in the Cayley graph $\Gamma(G,\AA)$.

\item A group $G$ is sofic if for any finite subset $K\subseteq G$ and any $\eps>0$, it admits a $(K,\eps)$-action $\psi:K\to\End(X)$ on a finite set $X$.

\item A group $G$ is sofic if for any finite subset $K\subseteq G$ and any $\eps>0$, it admits a $(K,\eps)$-action $\psi:K\to\End(X)$ on a finite set $X$ such that the uniform probability measure $\mu_X$ on $X$ is approximately invariant, in the sense that
\[
\|\mu_X-\psi(g)_*\mu_X\|\leq\eps
\]
for all $g\in K$, where $\|\cdot\|$ denotes the total variation norm.

\item A group $G$ is sofic if it acts essentially freely on a measure space $(X,\mu)$ equipped with an invariant, finitely additive probability measure $\mu$ defined on the set of all subsets of $X$.

%\item[(5)] A group $G$ with finite generating set $\AA$ is sofic if there exists a sequence $\{\Gamma_i\}_{i\in\N}$ of finite $\AA$-labeled graphs such that the sequence of measures $\{\mu_i\}_{i\in\N}$, where $\mu_i$ is the uniform probability measure on the set of rerootings of $\Gamma_i$, converges weakly to the Dirac measure concentrated on the Cayley graph of $G$ constructed with respect to $\AA$ in the space of probability measures on the space of rooted $\AA$-labeled graphs.

\end{enumerate}
\end{definition}

Definition~(1) is the original definition of Gromov. Note that, although it applies only to finitely generated groups, soficity may be regarded as a local property: by Definitions~(2) or (3), an arbitrary group is sofic if and only if all of its finitely generated subgroups are sofic. The condition of being approximable by finite labeled graphs is called the \emph{Weiss condition} in \cite{CSC}. Definition~(2)\ is due to Elek and Szab\'{o}. Definition~(3) is a generalization of a definition of an amenable group: A group $G$ is \emph{amenable} if for any finite subset $K\subseteq G$ and any $\eps>0$, there exists a finite subset $X\subseteq G$ such that $\mu_X$ is approximately invariant, in the sense that
\[
\|\mu_X-g_*\mu_X\|\leq\eps
\]
for all $g\in K$. Finally, Definition~(4) is again due to Elek and Szab\'{o} \cite{ES}. It too generalizes a known definition of an amenable group: A finitely generated group $G$ is amenable if and only if it admits an invariant, finitely additive probability measure defined on the set of all subsets of $G$.

Definitions~(1)-(4) do not represent an exhaustive list of definitions of a sofic group, but they are such that any one of them readily generalizes to a definition of a sofic monoid. Accordingly, we formulate the following definition.
\begin{definition}
We say that a monoid is ($k$)-\emph{sofic} if it is sofic according to the definition obtained by replacing the word ``group'' with the word ``monoid'' in part ($k$) of Definition~\ref{soficgroup} above.
\end{definition}
A monoid is sofic in the sense of Ceccherini-Silberstein and Coornaert if it is (2)-sofic, and one might guess that, given that they are equivalent for groups, Definitions~(1)-(4) will turn out to be equivalent for monoids as well. This is not the case, however. We will show that (1)-soficity and (2)-soficity are indeed equivalent, thereby clarifying an issue raised by Ceccherini-Silberstein and Coornaert, but that (1)-soficity and (2)-soficity are not equivalent to (3)-soficity, which implies (4)-soficity.

We do not wish to dwell on the question of which of these definitions, if any, yields the ``right'' notion of soficity for monoids. We feel that this is a matter of taste, and it is not our aim to tout one definition as right and another as wrong. We do wish, however, to explore these varying definitions in some detail, and to point out how (3)-soficity and (4)-soficity align more naturally with soficity as it is defined for other structures.

\section{Relating different notions of soficity for monoids}

Ceccherini-Silberstein and Coornaert show (see Theorem~6.1 of \cite{CSC}) that (1)-soficity implies (2)-soficity for monoids, but they establish the converse only under the assumption that the monoid is left-cancellative. Our first result is that equivalence holds in full generality.

\begin{theorem}\label{1iff2}
A monoid is (1)-sofic if and only if it is (2)-sofic.
\end{theorem}
\begin{proof}

As mentioned above, it is established in \cite{CSC} that (1)-soficity implies (2)-soficity. Conversely, let $M$ be a finitely generated (2)-sofic monoid, $\AA$ a finite generating set of $M$, and $M_{r,\AA}$ the $r$-neighborhood of the identity in the Cayley graph of $M$ constructed with respect to $\AA$.

Observe that it is sufficient to treat the case when $r=1$, since it is always possible to pass to a larger generating set. That is, constructing a graph that locally looks like $M_{r,\AA}$ is tantamount to constructing a graph that locally looks like $M_{1,\BB}$, where $\BB$ is, say, the generating set consisting of all monoid elements in $M_{r,\AA}$.

Accordingly, let $K$ consist of all monoid elements in $M_{1,\AA}$, and let $\psi:K\to\End(X)$ be a $(K,\eps)$-action of $M$ on a finite set $X$. Note that $\AA\subseteq K$. For simplicity, write $\psi_s\colonequals\psi(s)$, and endow $X$ with a graph structure by taking as its set of edges all triples of the form $(x,s,\psi_s(x))$, where $s\in K$. The edge $(x,s,\psi_s(x))$ is understood to be directed from $x$ to $\psi_s(x)$ and labeled with $s$. Denote by $\mu_X$ the uniform probability measure on $X$.

By Proposition~6.2 of \cite{CSC}, we may assume that $\psi_e=\Id_X$. For distinct elements $s,t\in K$, define the subset $A_{s,t}\subseteq X$ as
\[
A_{s,t}\colonequals\{x\in X\mid\psi_s(x)=\psi_t(x)\}.
\]
For elements $s,t\in K$ such that $st\in K$, define the subset $A_{s,t,st}\subseteq X$ as
\[
A_{s,t,st}\colonequals\{x\in X\mid(\psi_s\circ\psi_t)(x)\neq\psi_{st}(x)\}.
\]
By the definition of a $(K,\eps)$-action, each of the sets $A_{s,t}$ and $A_{s,t,st}$ has measure less than or equal to $\eps$, so that
\begin{equation}\label{bound1}
\mu_X\left(\bigcup_{s\neq t}A_{s,t}\right)\leq|K|^2\eps
\end{equation}
and
\begin{equation}\label{bound2}
\mu_X\left(\bigcup_{s,t,st\in K}A_{s,t,st}\right)\leq|K|^2\eps\equalscolon\frac{\delta}{2}.
\end{equation}
Denoting by $A$ the union of all sets $A_{s,t}$ and $A_{s,t,st}$, it follows from the bounds (\ref{bound1}) and (\ref{bound2}) that $\mu_X(X\backslash A)\geq1-\delta$. Clearly, $1-\delta$ can be made arbitrarily close to $1$ by choosing $\eps$ to be sufficiently close to $0$.

%Given a point $x\in X\backslash A$, we would like to exhibit an isomorphism $\phi:M_{1,\AA}\to B_1(x)$ such that $\phi(e)=x$, where $B_1(x)$ denotes the $1$-neighborhood of $x$. Given an edge $(e,s,s)$ in $M_{1,\AA}$, there exists an edge $(x,s,\psi_s(x))$ in $B_1(x)$. Since $x\notin A_{s,t}$, we find that $\psi_s(x)\neq\psi_t(x)$ for all distinct $s,t\in K$, showing that the edges $(x,s,\psi_s(x))$ and $(x,t,\psi_t(x))$ have distinct terminal points whenever $s\neq t$, and moreover, since $\psi_e=\Id_X$, that $x=\psi_s(x)$ if and only if $s=e$.

Given a point $x\in X\backslash A$, we would like to exhibit an isomorphism $\phi:M_{1,\AA}\to B_1(x)$ such that $\phi(e)=x$, where $B_1(x)$ denotes the $1$-neighborhood of $x$. If $x\in X\backslash A$, then consider the graph homomorphism $\phi:M_{1,\AA}\to B_1(x)$ given by
\[
\phi(s,t,ts)=(\psi_{s}(x),t,(\psi_t\circ\psi_s)(x))=(\psi_s(x),t,\psi_{ts}(x)).
\]
Let $(s,t,ts)$ and $(s',t',t's')$ be two edges in $M_{1,\AA}$. By construction, if $s\neq s'$, then $\psi_s(x)\neq\psi_{s'}(x)$. Likewise, if $ts\neq t's'$, then $\psi_{ts}(x)\neq\psi_{t's'}(x)$, which shows that $\phi$ is an embedding. It remains to show that $\phi$ is in fact surjective. To this end, note that it is impossible that there exist an edge $(x,s,\psi_s(x))$ in $B_1(x)$ which is not in the image of $\phi$, since for every $x\in X$, there is exactly one $s$-labeled outgoing edge attached to $x$ and it follows from the definition of $\phi$ that this edge is $\phi(e,s,s)$. Suppose $(y,s,\psi_s(y))$, where $y\neq x$, is some other edge in $B_1(x)$ which is not in the image of $\phi$. Since $y$ belongs to the $1$-neighborhood of $x$, there exists some $t\in K$ such that
\[
(x,t,\psi_t(x))=(x,t,y)
\]
is an edge in $B_1(x)$. But then
\[
(x,st,\psi_{st}(x))=(x,st,(\psi_s\circ\psi_t)(x))=(x,st,\psi_s(y))
\]
is an edge in $B_1(x)$ as well. The definition of $\phi$ and the fact that $y$ has a unique $s$-labeled outgoing edge attached to it now imply that $(y,s,\psi_s(y))$ must have belonged to the image of $\phi$ after all.
%If $(s,t,ts)$ is an edge in $M_{1,\AA}$, where $s\neq e$, then if $x\in X\backslash A$, there exists an edge
%\[
%(\psi_s(x),t,(\psi_t\circ\psi_s)(x))=(\psi_s(x),t,\psi_{ts}(x))
%\]
%in $B_1(x)$, which shows that $M_{1,\AA}$ embeds into $B_1(x)$. But this embedding is in fact an isomorphism: suppose 
%for distinct elements $s,t\in K$, the edges $(x,s,\psi_s(x))$ and $(x,t,\psi_t(x))$ have distinct endpoints.
%Let $x\in X\backslash A$, and define a map $\phi_x:K\to X$ as $\phi_x(s)\colonequals f_s(x)$. Note that we have $\phi_x(e)=\psi_e(x)=\Id_X(x)=x$. By construction, if $s,t\in K$ are distinct, then $\phi_x(s)\neq\phi_x(t)$. 
\end{proof}

%Given a $(K,\eps)$-action $\psi:K\to\End(X)$ of a semigroup $S$, one may apply the technique of \emph{amplification}, which works as follows. Consider the $n$-fold Cartesian product $X^n$, together with the diagonal embedding $\Image(\psi)\hookrightarrow\End(X)$ under which a given endomorphism $f\in\Image(\psi)$ is sent to the map $f:X^n\to X^n$ defined as
%\[
%f(x_1,\ldots,x_n)=(f(x_1),\ldots,f(x_n)).
%\]
%\begin{lemma}
%Let $f,g\in\End(X)$ be distinct endomorphisms of a finite set $X$. Then
%\[
%\lim_{n\to\infty}\mathbb{P}\left(f(\overline{x})\neq g(\overline{x})\mid\overline{x}\in X^n\right)=1,
%\]
%where the above probabilities are computed with respect to the uniform probability measure on $X^n$.
%\end{lemma}

Our next goal is to demonstrate that (3)-soficity for monoids is a very strong condition: a monoid is (3)-sofic if and only if it embeds into a group. Before proving this claim, we establish a lemma which asserts, roughly speaking, that under certain conditions a labeled graph which always looks the same when looking forwards must always look the same when looking backwards. To be more precise, if $\Gamma$ is a graph whose edges are directed and labeled with elements of an alphabet $\AA$ (an \emph{$\AA$-labeled graph}), denote by $\overline{\Gamma}$ the graph obtained by keeping all edge labels of $\Gamma$ the same but reversing the direction of each edge. Given a vertex $x\in\Gamma$, let
\[
B_\infty^+(\Gamma,x)=\varinjlim B_r(\Gamma,x),
\]
and let
\[
B_\infty^-(\Gamma,x)=B_\infty(\overline{\Gamma},x).
\]
As before, if there is no confusion about the graph we are referring to, we may write $B_\infty^+(x)$ instead of $B_\infty^+(\Gamma,x)$, etc. Note that if $\Gamma$ is the Cayley graph (or, more generally, a Schreier graph) of a group constructed with respect to a symmetric generating set $\AA$, then $\Gamma=\overline{\Gamma}$.

\begin{lemma}\label{Cayley}
Let $\Gamma$ be a connected $\AA$-labeled graph such that for each $x\in\Gamma$ and each $a\in\AA$, there exists precisely one outgoing and one incoming edge labeled with $a$ attached to $x$, and such that for any two vertices $x,x'\in\Gamma$, the neighborhoods $B_\infty^+(x)$ and $B_\infty^+(x')$ are isomorphic. Then $\Gamma$ is the Cayley graph of a group.
\end{lemma}
\begin{proof}
Let $x,x'\in\Gamma$ be two vertices. By assumption, there exists an isomorphism $\phi:B_\infty^+(x)\to B_\infty^+(x')$ of rooted labeled graphs. If $y\in B_\infty^+(x)$, put $y'\colonequals\phi(y)$. We claim that $\phi$ extends to an isomorphism
\[
\phi:B_\infty^+(x)\cup B_\infty^-(y)\to B_\infty^+(x')\cup B_\infty^-(y').
\]
Suppose it did not. Then there would exist a vertex $z\in B_\infty^-(y)$, obtained by starting at $y$ and following a word $w$ in the alphabet $\AA^{-1}$, together with distinct paths $\gamma_1$ and $\gamma_2$ which begin at $z$ and terminate at the same point, whereas the point $z'\in B_\infty^-(y')$, reached by starting at $y'$ and following the word $w$, would be such that the corresponding paths $\gamma_1'$ and $\gamma_2'$ did not terminate at the same point. Either this, or the same situation would occur with the roles of $B_\infty^-(y)$ and $B_\infty^-(y')$ reversed. In either case, it is obvious that $B_\infty^+(z)$ and $B_\infty^+(z')$ would not be isomorphic, a contradiction. It follows by induction that $\phi$ extends to an isomorphism
\[
\phi:\bigcup_{y\in B_\infty^+(x)}B_\infty^-(y)\to\bigcup_{y\in B_\infty^+(x')}B_\infty^-(y)
\]
between the full backwards orbits of $B_\infty^+(x)$ and $B_\infty^+(x')$.

Suppose next that $x_0\in B_\infty^+(x)$ is a \emph{deficient} vertex, i.e.\ a vertex such that $y\colonequals a^{-1}(x_0)\notin B_\infty^+(x)$ for some $a\in\AA$. Put $y'\colonequals a^{-1}(\phi(x_0))$. We claim that $\phi$ also extends to an isomorphism
\[
\phi:B_\infty^+(x)\cup B_\infty^+(y)\to B_\infty^+(x')\cup B_\infty^+(y').
\]
If it did not, then there would exist a word $w$ in the alphabet $\AA$ such that the path obtained by beginning at $y$ and following $w$ terminates at a vertex $z\in B_\infty^+(x)\cap B_\infty^+(y)$ but the path obtained by beginning at $y'$ and following $w$ terminates at a vertex $t'\notin B_\infty^+(x')$. Either this, or the same situation would occur with the roles of $B_\infty^+(y)$ and $B_\infty^+(y')$ reversed. Now let $w'$ be a word in the alphabet $\AA$ such that the path obtained by starting at $y'$ and following $w$ terminates at $z'\colonequals\phi(z)$ (since $z'\in B_\infty^+(y')$, such a path exists). We must then conclude that $B_\infty^+(y)$ and $B_\infty^+(y')$ are not isomorphic: The paths obtained by beginning at $y$ and following $w$ and $w'$, respectively, terminate at the same point $z$, whereas the corresponding paths beginning at $y'$ terminate at distinct points $t'\neq z'$. This is a contradiction.

It follows, finally, that by iterating the previous arguments, $\phi$ can be extended to a full isomorphism $\phi:(\Gamma,x)\to(\Gamma,x')$. One can first extend $\phi$ to an isomorphism between the backwards orbits of $B_\infty^+(x)$ and $B_\infty^+(x')$, then adjoin the forward orbits of a deficient vertex and extend $\phi$ further, and so on. But a connected $\AA$-labeled graph with the property that any two choices of root yield isomorphic rooted labeled graphs is necessarily the Cayley graph of a group.
\end{proof}

The next theorem shows that (3)-soficity is a very strong condition.

\begin{theorem}\label{embed}
A monoid is $(3)$-sofic if and only if it embeds into a sofic group.
\end{theorem}
\begin{proof}
Let $M$ be a (3)-sofic monoid and $\psi:K\to\End(X)$ a $(K,\eps)$-action of $M$ on a finite set $X$ such that $\|\mu_X-(\psi_s)_*\mu_X\|\leq\eps$ for all $s\in K$ (we again set $\psi_s\colonequals\psi(s)$ for notational convenience). Suppose that, for a given $s\in K$, there exists a set $A\subseteq X$ of measure $\mu_X(A)>\eps$ such that $|\psi_s^{-1}(x)|>1$ for all $x\in A$. Then
\begin{align*}
\|\mu_X-(\psi_s)_*\mu_X\|&=\max_{E\subseteq X}|\mu_X(E)-((\psi_s)_*\mu_X)(E)|\\
&=\max_{E\subseteq X}|\mu_X(E)-\mu_X(\psi_s^{-1}(E))|\\
&\geq|\mu_X(A)-\mu_X(\psi_s^{-1}(A))|\\
&\geq|\mu_X(A)-2\mu_X(A)|\\
&=\mu_X(A)>\eps,
\end{align*}
a contradiction. By the same argument, if there exists a set $A\subseteq X$ of measure $\mu_X(A)>\eps$ such that $|\psi_s^{-1}(x)|<1$ for all $x\in A$, then we would have $\mu_X(\psi_s^{-1}(A))=\mu_X(\emptyset)=0$, again allowing us to deduce that $\|\mu_X-(\psi_s)_*\mu_X\|>\eps$. It follows that for each $s\in K$, $|\psi_s^{-1}(x)|=1$ for all $x$ over a set of measure at least $1-\eps$, i.e.\ each $\psi_s$ is approximately invertible.

Accordingly, there exists a subset $X_0\subseteq X$ of measure at least $1-|K|\eps$ (which we can of course choose to be arbitrarily close to $1$ by choosing $\eps$ to be sufficiently small) such that all $\psi_s$, where $s\in K$, satisfy $|\psi_s^{-1}(x)|=1$ for any $x\in X_0$. Now endow $X$ with the labeled graph structure whose edges consist of all triples of the form $(x,a,(\psi_n)_a(x))$, where $a\in\AA\subseteq K$ belongs to a generating set of $M$. By Theorem~\ref{1iff2} (note that (3)-soficity immediately implies (2)-soficity), there exists a subset $X_1\subseteq X$ of measure arbitrarily close to $1$ such that for each $x\in X_1$, the $r$-neighborhood $B_r^+(x)$ is isomorphic to $M_{r,\AA}$. Thus, for all $x\in X_0\cap X_1$, which again has measure arbitrarily close to $1$, $B_r^+(x)\cong M_{r,\AA}$ and $x$ has exactly one outgoing and one incoming edge labeled with $s$ attached to it for all $s\in K$.

Let $\{\psi_n:K_n\to\End(X_n)\}_{n\in\N}$ be a sequence of $(K_n,\eps_n)$-actions such that each $K_n$ and each $X_n$ is finite,
\[
K_1\subseteq\ldots\subseteq K_n\subseteq\ldots,
\]
$\bigcup_{n\in\N}K_n=M$, and $\eps_n\to0$. We also choose our sequence in such a way that each of the uniform probability measures $\mu_n$ is approximately invariant. By the preceding argument, $\mu_n$ is not only approximately invariant with respect to each $(\psi_n)_s\in\End(X_n)$, where $s\in K_s$, but approximately invariant with respect to each $(\psi_n)^{-1}_s$. The weak limit $\mu$ of the measures $\mu_n$ is therefore invariant with respect to the free group generated by $\AA$. Moreover, $\mu$ is concentrated on rooted graphs $(\Gamma,x)$ such that $B_\infty^+(x)\cong M_{\infty,\AA}$ and such that $x$ has exactly one incoming edge labeled with $a$ attached to it for each $a\in\AA$. It follows from Lemma~\ref{Cayley} that $\mu$ is in fact concentrated on the Cayley graph of a group $G$, which is obviously a group into which $M$ embeds.
\end{proof}

\begin{corollary}
The notion of $(1)$-soficity (or $(2)$-soficity) is not equivalent to $(3)$-soficity for monoids.
\end{corollary}
\begin{proof}
Let $M$ be a monoid which does not embed into a group. By Proposition~4.7 of \cite{CSC}, the monoid $\hat{M}\colonequals M\cup\{e'\}$ obtained by adjoining a new element $e'$ to $M$ and defining $e'$ to be the identity is (2)-sofic and hence (1)-sofic. Since $M$ does not embed into a group, $\hat{M}$ does not embed into a group either. Therefore, $\hat{M}$ is not (3)-sofic.
\end{proof}

\begin{corollary}
If a monoid is $(3)$-sofic, then it is $(4)$-sofic.
\end{corollary}
\begin{proof}
Let $M$ be a (3)-sofic monoid. Then $M$ embeds into a sofic group $G$, which, by virtue of being sofic, admits an essentially free action $G\circlearrowright(X,\mu)$ on a measure space equipped with a finitely additive probability measure $\mu$ defined on the set of all subsets of $X$. The restriction of this action to $M$ yields the desired essentially free action of $M$.
\end{proof}
\begin{remark}
We do not know whether $(4)$-soficity also implies $(3)$-soficity but conjecture a positive answer.
\end{remark}
A semigroup $S$ need not have an identity. As such, carrying the notion of (1)-soficity (or, equivalently, (2)-soficity) over to semigroups is problematic: there is no canonical point in the Cayley graph of $S$ whose neighborhood is to serve as a model for approximation. In light of Theorem~\ref{embed}, however, (3)-soficity can be applied to semigroups without an identity element and is indeed closed under the taking of monoid subsemigroups. The result is arguably an unsatisfying definition, as soficity for semigroups then turns out to be no more general a notion than soficity for groups: if a group $G$ is sofic, then any subsemigroup of $G$ is sofic, and, conversely, if a semigroup is sofic, then it embeds into a sofic group. But here we reiterate our point that (3)-soficity, like notions of soficity for other objects, retains a clear connection with ergodic theory.

%Recall that an \emph{ultrafilter} $\UU$ on a set $X$ is a finitely additive probability measure $\UU:2^X\to\{0,1\}$ defined on the set of all subsets of $X$. An ultrafilter is \emph{nonprincipal} if it is nonatomic, i.e.\ $\UU(x)=0$ for every point $x\in X$.
%
%\begin{proposition}
%TBA
%\end{proposition}
%\begin{proof}
%Suppose the monoid $M$ is ($k$)-sofic. Let $\{K_m\}_{m\in\N}$ be an enumeration of the finite subsets of $M$, let $\{\eps_n\}_{n\in\N}$ be a sequence of positive numbers such that $\eps_n\to0$, and let $\psi_{m,n}:K_m\to\End(X_{m,n})$ be a
%$(K_m,\eps_n)$-action of $M$ on a finite set $X_{m,n}$. Then consider the disjoint union
%\[
%X\colonequals\bigsqcup_{m,n\in\N}X_{m,n}.
%\]
%There is a natural partial order on the system of sets $\{X_{m,n}\}_{m,n\in\N}$. Say that $X_{m,n}\leq X_{m',n'}$ if $K_m\subseteq K_{m'}$ and $\eps_{n'}\leq\eps_n$ (that is, $X_{m',n'}$ is acted upon by a larger set of monoid elements, with smaller error). Next, choose any nonprincipal ultrafilter $\UU$ on $X$ that contains (assigns full measure to) each of the cones
%\[
%C_{m,n}=\{X_{m',n'}\mid X_{m',n'}\geq X_{m,n}\},
%\]
%and for an arbitrary subset $A\subseteq X$, define
%\[
%\mu(A)\colonequals\lim_{X_{m,n}\to\UU}\frac{|A\cap X_{m,n}|}{|X_{m,n}|}.
%\]
%\end{proof}

\section{A notion of soficity for general semigroup actions}

Although we have found only a rather restrictive definition of soficity for semigroups, it is natual to pass from semigroups to their actions, whereupon a new idea presents itself. This idea has already played out in the context of groups, where Elek and Lippner developed the notion of a sofic discrete measured equivalence relation \cite{EL}. This notion might equally well be called a \emph{sofic dynamical system}, meaning a countable group acting by measure-preserving automorphisms on a Lebesgue space (recall that a \emph{Lebesgue space}---also called a \emph{standard probability space}---is a probability space whose nonatomic part is isomorphic to the unit interval equipped with Lebesgue measure). We will, in turn, consider a more general kind of dynamical system wherein groups are replaced by semigroups.
\begin{definition}
By a \emph{dynamical system}, we will mean a countable semigroup $S$ acting by measure-preserving endomorphisms on a Lebesgue space $(X,\mu)$. We will work only with left actions (it is of course possible to develop an analogous theory for right actions) and also require that the action of our semigroup be \emph{finite-to-one}, meaning that for any $s\in S$ and almost every $x\in X$, the preimage $s^{-1}(x)$ is finite.
\end{definition}
We would like to define the notion of a sofic approximation to a system $(X,\mu,S)$. In doing so, it is helpful to be able to visualize the action of the semigroup $S$. We do this in the following way. Given a point $x\in X$ and an element $s\in S$, we define the \emph{preimage graph} $\Gamma_x(s)$ of $x$ with respect to $s\in S$ to be the Schreier graph of the action of $s$ restricted to the set $\{x\}\cup s^{-1}(x)$. Thus, $\Gamma_x(s)$ is the graph whose vertex set is $\{x\}\cup s^{-1}(x)$ and whose edge set consists of all triples $(y,s,x)$, where $y\in s^{-1}(x)$. We define a \emph{partial Schreier graph} to be a connected union of preimage graphs. As the set of partial Schreier graphs is naturally ordered by inclusion, we define the \emph{full Schreier graph} of $x$, denoted $\Gamma_x$, to be the maximal partial Schreier graph containing $x$.
\begin{definition}
We denote by $\Sch(S,X)$ the \emph{space of full Schreier graphs} of the action of $S$ on $X$, taking each graph $\Gamma_x\in\Sch(S,X)$ to be rooted at the vertex $x$.
\end{definition}
Note that if $S$ is a group, then a given $\Gamma_x$ is the usual Schreier graph of $S$ acting on the orbit of the point $x$, but a full Schreier graph is in general a significantly larger object, taking into account, as it were, many different forward and backward orbits.

Now let $X$ be a standard Borel space. For the sake of concreteness, we will always assume that $X=\{0,1\}^\infty$. The space $X$ has a natural projective structure
\[
X=\varprojlim X_r,
\]
where $X_r=\{0,1\}^r$ and the connecting maps $\pi_r:X_r\to X_{r-1}$ restrict a binary string of length $r$ to its first $r-1$ digits.  An element $\omega\in X_r$ thus determines the \emph{cylinder set}
\[
C_\omega\colonequals\{x\in X\mid\pi_{\infty,r}(x)=\omega\},
\]
where $\pi_{\infty,r}:X\to X_r$ is the natural projection onto $X_r$, and we denote by $\AAA$ the corresponding Borel $\sigma$-algebra on $X$. Given a finite-to-one action of a countable semigroup $S$ on $X$, we may pass to the associated space of full Schreier graphs $\Sch(S,X)$, which we endow with a projective structure as follows. Let $\{S_i\}_{i\in\N}$ be an enumeration of $S$. Define $\Sch_r(S,X)$ to be the set of (isomorphism classes of) $r$-neighborhoods of the roots of full Schreier graphs $\Gamma_x\in\Sch(S,X)$ which are spanned by the elements $s_1,\ldots,s_r$ (all other edges are neglected) and whose vertex-labels are truncated to their first $r$ digits. It is easy to see that
\[
\Sch(S,X)=\varprojlim\Sch_r(S,X),
\]
where the connecting maps $\rho_r:\Sch_r(S,X)\to\Sch_{r-1}(S,X)$ are the obvious restriction functions. As before, an element $U\in\Sch_r(S,X)$ determines the cylinder set
\[
C_U\colonequals\{\Gamma_x\in\Sch(S,X)\mid\rho_{\infty,r}(\Gamma_x)=U\},
\]
where $\rho_{\infty,r}:\Sch(S,X)\to\Sch_r(S,X)$ is the natural projection onto $\Sch_r(S,X)$, and we denote by $\BBB$ the corresponding Borel $\sigma$-algebra on $\Sch(S,X)$. It should be pointed out that an element of $\Sch_r(S,X)$ is to be understood as an $r$-neighborhood taken with respect to the usual graph metrics on full Schreier graphs $\Gamma_x$.

Our idea is to pass from $X$ to $\Sch(S,X)$, a space which records all of the information about the action of $S$. The map $f:(X,\AAA)\to(\Sch(S,X),\BBB)$ given by $f(x)=\Gamma_x$ is clearly a bijection, and it is easy to see that $f^{-1}$ is measurable: Given a cylinder set $C_\omega\subseteq X$, where $\omega\in X_r$, let $\{U_i\}_{i\in I}$ be the collection of all partial Schreier graphs $U_i\in\Sch_r(S,X)$ whose roots are labeled with $\omega$. Then
\[
f(C_\omega)=\bigcup_{i\in I}C_{U_i},
\]
which is measurable (note that $I$ is at most countable, since the action of $S$ is finite-to-one). Let $\AAA'$ denote the completion of $\AAA$ with respect to the measure $\mu$. Then $f_*^{-1}\BBB\subseteq\AAA'$ (the image of a Borel set under a measurable function is analytic and hence Lebesgue measurable), which allows us to push forward the measure $\mu$, obtaining a measure $\nu\colonequals f_*\mu$ on $\Sch(S,X)$. Let $\BBB'$ denote the completion of $\BBB$ with respect to $\nu$.
\begin{proposition}
The map $f:(X,\AAA',\mu)\to(\Sch(S,X),\BBB',\nu)$ given by $f(x)=\Gamma_x$ is an $S$-equivariant isomorphism of Lebesgue spaces.
\end{proposition}
\begin{proof}
It is clear that $f$ is an $S$-equivariant bijection, since $s(\Gamma_x)=\Gamma_{s(x)}$. Moreover, $f$ is measurable: If $B\in\BBB$ is a Borel set, then as noted above, $f^{-1}(B)\in\AAA'$. If $B'\in\BBB'\backslash\BBB$, then $B'\subseteq B$, where $B$ is a Borel set with $\nu(B)=0$. But this implies that $\mu(f^{-1}(B))=0$. Since $f^{-1}(B')\subseteq f^{-1}(B)$, we find that $f^{-1}(B')\in\AAA'$. The exact same argument shows $f^{-1}$ to be measurable as well, which establishes the claim.
\end{proof}
Our goal is to approximate a given dynamical system $(X,\mu,S)$ with a finite system. Accordingly, suppose that $X'$ is a graph whose vertices are labeled with elements of $X$ (i.e.\ $X'$ is an \emph{X-labeled set}) and whose edges are directed and labeled with elements of $S$. If $\mu'$ is a probability measure on $X'$, then we claim that $\mu'$ naturally determines a probability measure $\nu'$ on $\Sch(S,X)$: Given an $r$-neighborhood $U\in\Sch_r(S,X)$, set
\[
\nu'(C_U)=\sum_{B_r(x)\cong U}\mu'(x)
\]
as long as there exists an $x\in X'$ such that $B_r(x)\cong U$ when the vertex-labels of $X'$ are truncated to their first $r$ digits. Once this requirement is met, $\nu'$ may be extended arbitrarily to other cylinder sets in such a way that it becomes a probability measure (the measure $\mu'$ therefore technically defines a family of measures on $\Sch(S,X)$, but the choice of representative will not matter for us). Note that if $K\subseteq S$ and $\psi:K\to\End(X')$ is a $(K,\eps)$-action on an $X$-labeled set $X'$, then $X'$ comes equipped with a natural graph structure, which consists of all triples $(x,s,\psi_s(x))$, where $x\in X'$ and $s\in K$.
\begin{definition}
Let $(X,\mu,S)$ be a dynamical system, let $K\subseteq S$, and let $\eps>0$. A triple $(X',\mu',\psi)$, where $X'$ is a finite $X$-labeled set, $\mu'$ is a probability measure on $X'$, and $\psi:K\to\End(X')$ is a $(K,\eps)$-action, is a \emph{$(K,\eps)$-approximation} to $(X,\mu,S)$ if
\[
\|\mu'-\psi(s)_*\mu'\|\leq\eps
\]
for all $s\in K$ and
\[
\left|\int f\,d\nu-\int f\,d\nu'\right|\leq\eps
\]
for any bounded continuous function $f:\Sch(S,X)\to\R$, where $\nu$ and $\nu'$ are the probability measures on $\Sch(S,X)$ determined by $(X,\mu,S)$ and $(X',\mu',\psi)$, respectively.
\end{definition}
We thus require the measure $\mu'$ on $X'$ to be approximately invariant with respect to $\psi$ and the measure $\nu'$ to be close, in the weak topology, to the $S$-invariant measure $\nu$.
\begin{definition}
A dynamical system $(X,\mu,S)$ is \emph{sofic} if for any finite $K\subseteq S$ and any $\eps>0$, it admits a $(K,\eps)$-approximation.
\end{definition}
It is worth pausing to point out that our approach mimics, in many ways, the approach of Elek and Lippner, which in turn 
mimics the idea behind so-called Benjamini-Schramm convergence, introduced in \cite{BS}. Here a sequence of graphs (or other objects---in particular, graphs which carry additional stucture) is interpreted as a sequence of probability measures on the space of rooted graphs by choosing the position of the root of each graph in the sequence uniformly at random. Each such measure is an invariant (or \emph{unimodular}) measure, and one may study the corresponding weak limit of the sequence. The key difference with our approach is that we consider general (approximately) invariant measures on finite structures, and not merely those which arise upon choosing the root uniformly at random. We conclude with a basic example of a sofic dynamical system.
\begin{example}
Denote by $\S^1=[0,1]/\!\!\sim$ the circle, obtained by gluing together the endpoints of the unit interval, and consider the classical angle-doubling map $f:\S^1\to\S^1$ given by
\[
f(x)=2x\,\,(\modulo 1).
\]
The map $f$ is $2$-to-$1$, and the Lebesgue measure $\lambda$ on $\S^1$ inherited from $[0,1]$ is invariant with respect to the $\N$-action determined by $f$. An element of $(\Sch(\N,X),\nu)$, the corresponding space of full Schreier graphs, is almost surely a rooted tree each of whose vertices has two incoming $f$-labeled edges attached to it and one outgoing $f$-labeled edge attached to it (here we have made an $\N$-equivariant identification between $\S^1$ and $X=\{0,1\}^\infty$ and, abusing notation, will also denote by $\lambda$ the image of the Lebesgue measure under this identification). To construct a sofic approximation to the system $(\S^1,\lambda,\N)$, let $X'$ be the disjoint union of all possible preimage trees of elements $\omega\in\{0,1\}^r$, up to depth $k$. That is, a connected component of $X'$ consists of an $X$-labeled vertex at level $0$, together with two $X$-labeled preimage vertices at level $1$, and so on up to level $k$, which consists of $2^k$ $X$-labeled preimage vertices. A connected component $U$ of $X'$ also determines a cylinder set $C_U$, which has a certain measure. Define a measure on $X'$ by assigning mass $2^{-k}\lambda(C_U)$ to each vertex at level $k$ of a connected component $U$ of $X'$, then normalize to obtain a probability measure $\lambda'$. It is not difficult to see that, provided $k$ is chosen to be sufficiently large, the measure $\lambda'$ is approximately invariant with respect to the elements
\[
\{f,\ldots,f^{\circ n}\}\subseteq\N
\]
and moreover that the associated measure $\nu'$ on $\Sch(\N,X)$ approximates $\nu$ in the weak topology. Choosing ever larger values of $k$ and carrying out this construction shows the system $(\S^1,\lambda,\N)$ to be sofic.
\end{example}

It would appear that there is much to investigate concerning our definition of a sofic dynamical system. Idle questions include: Is every $\N$-action sofic? (We conjecture a positive answer.) Is every dynamical system (in our sense) sofic? Is it possible to work only with the full orbit equivalence relation determined by a semigroup action, thereby developing a theory analogous to that of Feldman and Moore (see \cite{FM})? Suppose a dynamical system $(X,\mu,S)$ admits a \emph{natural extension} $(\widetilde{X},\widetilde{\mu},S)$ that makes each $s\in S$ invertible---what, if any, is the relationship between soficity of the original system and soficity of the extension? Which semigroups admit faithful actions by measure-preserving endomorphisms on a Lebesgue space?

A further idea is implicit in our work. Having relaxed the requirement that sofic approximations to a dynamical system come equipped with a uniform measure, it is natural to ask whether there is a reasonable way to approximate dynamical systems which come only with a quasi-invariant measure (so that the measure class, rather than the measure itself, is preserved), or even an arbitrary measure, perhaps by requiring the Radon-Nikodym derivatives associated to the finite approximating system to converge to those of the original system. It could be interesting to develop soficity in this direction as well and examine its relationship with soficity for dynamical systems with an invariant measure. We may pursue this idea in a future work.

\end{document}